\documentclass[12pt]{article}

\usepackage[tbtags]{amsmath}
\usepackage{amsfonts,amssymb,amsthm}
\newtheorem{theorem}{Theorem}
\newtheorem{prop}{Proposition}[section]
\newtheorem{lemma}[prop]{Lemma}
\newtheorem*{corollary}{Corollary}
\numberwithin{equation}{section}

\theoremstyle{definition}
\newtheorem{remark}[prop]{Remark}

\def\ba#1{\begin{array}{#1}}
\def\ea{\end{array}}

\newcommand{\dst}{\displaystyle}
\newcommand{\RR}{\mathbb{R}}
\newcommand{\NN}{\mathbb{N}}
\newcommand{\QQ}{\mathbb{Q}}
\newcommand{\ZZ}{\mathbb{Z}}
\newcommand{\ee}{\mathbf{e}}
\newcommand{\xx}{\mathbf{x}}
\newcommand{\rr}{\mathbf{r}}
\newcommand{\yy}{\mathbf{y}}

\newcommand{\eps}{\varepsilon}

\newcommand{\tT}{\tilde{T}}
\newcommand{\mend}{\mathrm{\kappa}}
\newcommand{\plt}{\ll} 

\newcommand{\intv}[2]{[{#1}:{#2}]}         
\newcommand{\ai}[1]{a_{#1}}         
\newcommand{\av}[1]{A(#1)}         
\newcommand{\mf}[1]{m^+_{#1}}  
\newcommand{\mcs}{S_{n,*}}          
\newcommand{\infS}[1]{A_{#1,*}}   
\newcommand{\cns}[1]{[\mathbf{#1}]}  

\newcommand{\minT}[1]{T_{#1,*}}   
\newcommand{\infT}{T_{*}}      
\newcommand{\mintT}[1]{\tT_{#1,*}}   
\newcommand{\inftT}{\tT_{*}}      

\usepackage{pgf,tikz,pgfplots}
\pgfplotsset{compat=1.15}
\usepackage{mathrsfs}
\usetikzlibrary{arrows}

\begin{document}

\author{Sergey Sadov\footnotemark[1]\footnote{E-mail: serge.sadov@gmail.com}}
\date{}

\title{Lower bound for cyclic sums with one-sided maximal averages in denominators}


\maketitle

\begin{abstract}
Let $\xx=(x_1,\dots,x_n)$ be an $n$-tuple of positive real numbers
and the sequence $(x_i)_{i\in\ZZ}$ be its $n$-periodic extension. 
Given an $n$-tuple $\rr=(r_1,\dots,r_n)$ of positive integers, let
$a_i$ be the arithmetic mean of $x_{i+1},\dots,x_{i+r_i}$.
We form the cyclic sums
$S_n(\xx,\rr)=\sum_{i=1}^n x_i/a_{i}$,
following the pattern of the long studied Shapiro sums, which correspond to all $r_i=2$,
and more general Diananda sums, where all $r_i$ are equal.
We find the asymptotics of the $\rr$-independent lower bounds $\infS{n}=\inf_{\rr}\inf_{\xx} S_n(\xx,\rr)$ as
$n\to\infty$: it is
$\infS{n}=e\log n - A+O(1/\log n)$.


\end{abstract}

\section{Notation, background, motivation} 
\label{sec:1problem}

Let $\xx=(x_1,\dots,x_n)$ be an $n$-tuple ($n\geq 1$) of nonnegative real numbers. 
The $n$-periodic extension of $\xx$ is the doubly-infinite sequence $(x_n)$ defined by
$x_{i+kn}=x_i$ ($k\in\ZZ$, $i=1,\dots,n$). 

An integer interval $I=[a:b]$ ($a,b\in\ZZ$, $b\geq a$) is the finite set $\{a,a+1,\dots,b-1,b\}$
of cardinality
$$
|I|=b-a+1.
$$
In particular, $[a:a]=\{a\}$ is an interval of cardinality one. 

Given an $n$-tuple $\xx$ and an interval $I=[a:b]$, consider the {\em interval average}\  
$$
 \ai{I}(\xx)=\frac{1}{|I|}  {\sum_{j\in I} x_j}=\frac{1}{b-a+1}\sum_{j=a}^b x_j.
$$
The $x_j$ in the right-hand side are members of the periodic extension of $\xx$. 

Let $\rr=(r_1,\dots,r_n)$ be an arbitrary $n$-tuple of positive integers. Introduce sums of the form
\begin{equation}
\label{varcsum}
S_n(\xx,\rr)=\sum_{i=1}^n\frac{x_i}{\ai{[i+1:\,i+r_i]}(\xx)}=\sum_{i=1}^n \left(\frac{x_i r_i}{\sum_{j=1}^{r_i} x_{i+j}}\right).
\end{equation}
We assume throughout that the denominators are nonzero. In this case we say that the pair  $(\xx,\rr)$ of $n$-tuples is {admissible}; if one of them  ($\xx$ or $\rr$) is fixed, we call the other $n$-tuple admissible.

The research question here is to determine the uniform in $(\xx,\rr)$ lower bounds for the sums \eqref{varcsum}, particularly ---their  large $n$ asymptotics. 
The question has its roots in the long known Shapiro problem. Let us sketch out the connection.

For $n$-tuples $\rr$ with equal components we use notation 
$$
\rr=(k,k,\dots,k)=\cns{k}.
$$
(The parameter $n$ is implicit; more pedantic notation would be $\cns{k}_n$.)

In the simplest case $\rr=\cns{1}$ we have
$$
 S_n(\xx,\cns{1})=\sum_{i=1}^n\frac{x_i}{x_{i+1}}
$$
and
$$
 \inf_{\xx} S_n(\xx,\cns{1})=n
$$
by the inequality between the arithmetic and geometric means (AM-GM).
Hereinafter the infimum is taken over all admissible $n$-tuples.

The case $\rr=\cns{2}$ involves  {\it Shapiro's sums}\  
$$
\frac{1}{2} S_n(\xx,\cns{2})=\sum_{i=1}^n\frac{x_i}{x_{i+1}+x_{i+2}}
$$
named after Harold S.~Shapiro who proposed \cite{Shapiro_1954} to prove that $\inf_{\xx} S_n(\xx,\cns{2})=n$.
It has been known since the late 1950s that Shapiro's conjecture is invalid  for sufficiently large $n$.
Drinfeld \cite{Drinfeld_1971} analytically determined the constant
$$
B_2 =\inf_{n\geq 1}\frac{1}{n}\inf_{\xx} S_n(\xx,\cns{2})=0.989133\dots.
$$
A review and history of results related to Shapiro's problem can be found in \cite[Ch.~16]{MPF_1993} or  \cite{Clausing_1992}.


The sums
\footnote{Different normalizations have been used. Diananda \cite{Diananda_1962a} designates as $B(x_1,\dots,x_n)$ ($k$ implicit) what we would write as $(kn)^{-1} S_{n,k}(\xx)$. In \cite{Sadov_2015} the sums 
\eqref{diansum} are denoted $S_{n,k}(\xx)$.} 
\begin{equation}
\label{diansum}
\frac{1}{k}S_n(\xx,\cns{k})=\sum_{i=1}^n\frac{x_i}{x_{i+1}+\dots+x_{i+k}}
\end{equation}
with arbitrary integer $k\geq 1$ were first considered by Diananda \cite{Diananda_1959}.
 In \cite{Diananda_1962a} Diananda showed 
that 
$$
\sum_{i=1}^n\frac{x_i}{\max(x_{i+1},\dots,x_{i+k})}\geq \frac{1}{k}\left\lfloor\frac{n+k-1}{k}\right\rfloor \geq \frac{n}{k^2}.
$$
Put 
\footnote{The present notation $A_{n,k}$, resp., $B_k$, corresponds to $kA(n,k)$ resp., $B(k)$, in \cite{Sadov_2015}.}
\begin{equation}
\label{Ank}
A_{n,k}=\inf_{\xx} S_n(\xx,\cns{k}).
\end{equation}
It follows 
that the constants
$$
B_k=\inf_{n\geq 1}\frac{A_{n,k}}{n}
$$
are strictly positive, namely, 
$$
 B_k\geq \frac{1}{k}.
$$
Diananda's estimate does not imply that they are uniformly bounded away from zero. That this is the case was 
proved by the author \cite{Sadov_2015}: the sequence $(B_k)$ is nonincreasing and its monotone limit  
$$
B_\infty=\downarrow\lim_{k\to\infty} B_k
$$
satisfies the estimates
$$
\log 2\leq B_\infty\leq 0.930498\dots.
$$
Exact values of $B_\infty$ and of $B_k$, $k\geq 3$, are as yet not known. 


Thus the cyclic sums with summands $x_i/\ai{[i+1:\,i+k]}(\xx)$, $i=1,\dots,n$, admit the lower bounds
\begin{equation}
\label{uniformconstk}
S_n(\xx,\cns{k})\geq B_\infty n
\end{equation}
uniformly in $k$. The (unknown) constant $B_\infty$ is best possible by definition. 
Let
$$
A_n=\inf_{k\geq 1}A_{n,k}.
$$
From \eqref{uniformconstk} on the one hand and from the special case $\xx=(1,\dots,1)$ on the other hand we infer the inequality
\begin{equation}
\label{An}
B_\infty n\leq A_n\leq n.
\end{equation}
It is a simple exercise to check that $A_n=n$ for $n=1,2,3$. Known results about Shapiro's problem tell us that $A_{n,2}<n$
when $n\geq 24$ or $n=2m$, $m\geq 7$. 
Therefore $A_n<n$ for such $n$.
The author does not know whether $n=14$ is the smallest for which $A_n<n$. (It is the smallest for which $A_{n,2}<n$).


\smallskip
Since, by \eqref{An}, the ratio $\inf_{k}\inf_{\xx} S_n(\xx,\cns{k})/n$
lies between two positive constants, it is natural to ask whether the greatest lower bound remains asymptotically linear in $n$ if the length of the interval of averaging in the denominators of Diananda sums is allowed to vary from summand to summand. 
This prompts us to consider sums of the form \eqref{varcsum}. 

Generalizing \eqref{Ank}, for $n\geq 1$ and a positive integer $n$-tuple $\rr$ we put
\begin{equation}
\label{Anrr}
A_{n,\rr}=\inf_{\xx} S_n(\xx,\rr)
\end{equation}
(where $\rr$ is fixed). Minimizing further with respect to $\rr$, define
\begin{equation}
\label{Anstar}
  A_{n,*}=\inf_{\rr} A_{n,\rr}, 
\end{equation}
where $\rr$ ranges over positive integer $n$-tuples. 
In particular, $\rr$ may equal to any constant $n$-tuple $\cns{k}$. So
\begin{equation}
\label{AnstarAn}
 A_{n,*}\leq A_n.
\end{equation}

Here comes 

\bigskip\noindent
{\bf Main question. }
Determine the asymptotic behavior of $A_{n,*}$ as $n\to\infty$.

\bigskip
The answer is given in Theorem~\ref{thm:main} below.

We conclude this section with an asymptotically precise version of the estimate \eqref{An}. 

\begin{prop}
\label{prop:limAn}
There holds the identity
$$
\lim_{n\to\infty}\frac{A_n}{n}=B_\infty
$$
or, equivalently,
$$
\lim_{n\to\infty}\frac{1}{n}\inf_{k} A_{n,k} = \lim_{k\to\infty} \inf_{n}\frac{A_{n,k}}{n},
$$
where $A_{n,k}$ are defined in Eq.~\eqref{Ank}.
\end{prop}

\begin{proof}
Denote $B'=\limsup_{n\to\infty} n^{-1} A_n$. 
By \eqref{An}, $B'\geq B_\infty$. Suppose, by way of contradiction, that $B'>B_\infty$. Since $B_k\downarrow B_\infty$, there exists $k_0$ such that $B'>B_{k_0}$. 
Take some $B''$ such that $B_{k_0}<B''<B'$.
There is a sequence $n_j\to\infty$ such that $A_{n_j}\geq B'' n_j$ for all $j$. Consequently,
$$
 \frac{A_{n_j,k}}{n_j}\geq B''>B_{k_0}
$$
for any $j$ and $k$. 

But according to \cite[Theorem~4]{Sadov_2015}, $\lim_{n\to\infty} A_{n,k}/n=B_k$ for all $k$. 
Taking $k=k_0$ and passing to the limit as $j\to\infty$ in the above formula, we obtain a contradiction.
\end{proof}

\section{Sums with maximal forward averages in denominators}
\label{sec:summfa}

Before formulating the main result, we will give an alternative, more convenient expression for $A_{n,*}$. 

In \eqref{Anrr}--\eqref{Anstar} we have defined $A_{n,*}$ as an ``iterated infinum'' (mimicking familiar terminology of integral calculus)
$A_{n,*}=\inf_{\rr}\inf_{\xx} S_n(\xx,\rr)$. 
Obviously, it can be equivalently written as a ``double infimum'' 
$$
 A_{n,*}=\inf_{(\xx,\rr)}S_{n}(\xx,\rr)
$$
over all admissible pairs of $n$-tuples. It better suits our purpose to define $A_{n,*}$ as an iterated infimum in the other order.  

Introduce the function
$$
  \mcs(\xx)=\inf_{\rr} S_n(\xx,\rr),
$$
where $\rr$ ranges over all admissible integer $n$-tuples
(for fixed $\xx$). 
Then
\begin{equation}
\label{infS}
\infS{n}=\inf_{\xx}\inf_{\rr} S_{n}(\xx,\rr)=\inf_{\xx}\mcs(\xx).
\end{equation}

We will write the function $\mcs(\xx)$ as a cyclic sum. Introduce the {\em maximal forward averages} 
\begin{equation}
\label{mfavg}
 \mf{i}(\xx)=\sup_{r\geq 1} \ai{[i+1:\,i+r]}(\xx).
\end{equation}
As a matter of fact, `$\sup$' can be replaced by `$\max$'
and the maximization range can be restricted to $1\leq r\leq n$, see Proposition~\ref{prop:len-max-int_le-n} below.

 It is easy to see that
\begin{equation}
\label{maxcsum2}
 \mcs(\xx)=\sum_{i=1}^n \frac{x_i}{\mf{i}(\xx)}.
\end{equation}
This last form emphasizes the analogy between $\mcs(\xx)$ and the classical Shapiro sums or Diananda sums \eqref{diansum}. 

The {\em right maximal function} $\xx\mapsto M^r\xx$ of Hardy and Littlewood (cf.\ e.g.\ \cite[Theorem 3.4]{Phillips_1967}) for the periodically extended sequence $(x_i)$ is
\footnote{The superscript `$r$' in $M^r$ stands for `right'} 
\begin{equation}
\label{MaxFr}
 M^r\xx(i)=\max_{r\geq 1} \ai{[i:\,i+r-1]}(\xx),
\end{equation}
One recognizes the relation
$\mf{i}(\xx)=M^r \xx(i+1)$. 

\begin{prop}
\label{prop:len-max-int_le-n}
In \eqref{mfavg} and in \eqref{MaxFr} it suffices to maximize over $r\leq n$. 
\end{prop}

\begin{proof}
Suppose that 
$r>n$ and write $r=n\ell+r'$, where $1\leq r'\leq n-1$;
then by $n$-periodicity
$$
r\ai{[i:i+r-1]}(\xx)=r'\ai{[i:i+r'-1]}(\xx)
+(r-r')\ai{[1:n-1]}(\xx).
$$
It follows that $\ai{[i:i+r-1]}\leq \max\left(\ai{[i:i+r'-1]},
\ai{[i:i+n-1]}\right)$.
\end{proof}

From the double-sided estimate for $\mf{i}(\xx)$
$$
\max_j x_j\geq \mf{i}(\xx)\geq \ai{[i+1:\,i+n]}(\xx)=\ai{[1:n]}(\xx)
$$
it follows that for any $\xx$
$$
 n\geq \mcs(\xx)\geq 1 .
$$
The same bounds hold of course for $\infS{n}$. 
In view of \eqref{AnstarAn} and Proposition~\ref{prop:limAn}
the upper bound can be asymptotically improved:
$$
 \limsup_{n\to\infty} \frac{\infS{n}}{n}\leq B_\infty.
$$
On the other hand, it is not immediately clear why the sequence $\infS{n}$ should be unbounded. 

 \smallskip
Our main result asserts that $\infS{n}=\inf\mcs(\xx)$ does indeed grow with $n$, but the growth is logarithmic,
in sharp contrast with (\ref{An}).

\begin{theorem}
\label{thm:main}
There holds the asymptotic formula
$$
\infS{n}= e\log n - A+O\left(\frac{1}{\log n}\right)
$$
as $n\to\infty$. The numerical value of the constant $A$ is
$$
 A=1.70465603718\dots.
$$
\end{theorem}

We will prove Theorem~\ref{thm:main} by ``uncycling'' the cyclic sum \eqref{maxcsum2} and reducing the problem to an optimization problem which resembles the one associated with the inequality between the arithmetic and geometric means (AM-GM). 

Introduce the functions
\begin{equation}
\label{tTN}
 \tT_n(\xx,p)=\sum_{i=1}^{n-1} \frac{x_i}{x_{i+1}}+\frac{x_n}{p}.
\end{equation}
(The notation $\tT$ will appear more justified in Sec.~\ref{sec:reduction}.)

In the AM-GM problem we seek to minimize $\tT_n(\xx,p)$
with given $p$ under the constraint (say) $x_1=1$.
The problem of interest here involves a different normalization constraint. Define
\begin{equation}
\label{mintTN}
\mintT{n}(p)=\inf_{\xx\geq 0,\;\sum_{1}^n x_i=1} \tT_n(\xx,p).
\end{equation}

\begin{prop}
\label{prop:reduction}
There holds the equality
$$
\infS{n}=\mintT{n}\left(\frac{1}{n}\right).
$$
\end{prop}

As a matter of fact, the reduced problem \eqref{mintTN} 
is considerably more difficult than AM-GM. We present a complete solution in a separate paper \cite{Sadov_2022logper}.
Its main result, from which Theorem~\ref{thm:main} immediately follows, is the asymptotic formula
\footnote{In \cite{Sadov_2022logper} the asymptotics is
written in an even more precise form, with a specific expression for the $O(1/|\log p|)$ term.}
$$
\inftT(p)= e|\log p| - A+O\left(\frac{1}{|\log p|}\right)
\quad\text{as $p\to 0^+$.}
$$

\underline{Our task here will be} to carry out the reduction
of the problem \eqref{infS} to the problem \eqref{mintTN},
i.e.\ \underline{to prove Proposition~\ref{prop:reduction}.
}
This is done in Sections~\ref{sec:maxint}--\ref{sec:reduction}.


\medskip
The next result should alert the reader about  misconstrued generalizations. For instance, it shows that
for the cyclic sums 
$$
 \sum_{i=1}^n \frac{x_i}{M^r\xx(i)}
$$
that look similarly to \eqref{maxcsum2} the lower bound equals $1$ for all $n$.

\begin{prop}
\label{prop:overgeneral}
Suppose that to any $i\in [1:n]$ there is assigned a collection $\mathcal{S}_i=\{\Omega_{i,1},\dots,\Omega_{i,k_i}\}$ of
subsets of $[1:n]$ and define
$$
 m_i(\xx)=\max_{1\leq j\leq k_i} a(\xx|\Omega_{i,j}),
$$
where $a(\xx|\Omega)$ is the arithmetic mean of the numbers
$x_j$, $j\in\Omega$. 

If $\mathcal{S}_i$ contains  $[1:n]$ for all $i$ and $\mathcal{S}_{i_0}\ni\{i_0\}$ for some $i_0$, then
$$
 \inf_{\xx} \sum_{i=1}^n \frac{x_i}{m_i(\xx)}=1.
$$
\end{prop}

\begin{proof}
Let $M=\max_i x_i$. 
 Note that $m_i(\xx)\leq M$ for all $i$, hence
$$
 \sum_{i=1}^n \frac{x_i}{m_i(\xx)}\geq\frac{1}{M} \sum_{i=1}^n x_i\geq \frac{M}{M}=1.
$$
Without loss of generality we may assume that $i_0=1$. Consider the vector $\xx$ with components $x_1=1$, $x_i=\eps$ for $i=2,\dots,n$.
If $\eps$ is small, then $m_1=1$ (because $\{1\}\in\mathcal{S}_1$). On the other hand,
$[1:n]\in\mathcal{S}_i$ for all $i$, so
$m_i\geq A=(x_1+\dots+x_n)/n>1/n$.

Hence for the so chosen $\xx$ 
$$
 \sum_{i=1}^n \frac{x_i}{m_i(\xx)}\leq 1+\sum_{i=2}^n
\frac{\eps}{1/n}=1+(n-1)n\eps.
$$
Since $\eps$ can be arbitrarily small, the result follows.
\end{proof}

Our sums $\mcs(\xx)$ are a particular case $(\ell=1)$
of the cyclic sums of the form
$$
  \sum_{i=1}^n \frac{x_i}{M^r\xx(i+\ell)}
$$
with fixed $\ell$. For $n\geq\ell\geq 1$ these do not fall under the conditions of Proposition~\ref{prop:overgeneral} and the minimization is nontrivial to at least the same degree as it is in the present paper.
As a further step in the study of cyclic sums inspired by Shapiro's problem one can ask about the uniform in $\ell>1$ asymptotics of their lower bounds as $n\to\infty$.


\section{Maximal intervals}
\label{sec:maxint}

Let us introduce terminology pertaining to index intervals.
Here $\xx$ is fixed and it is treated as an infinite $n$-periodic sequence.

\smallskip
Intervals of the form $[i:k]$ and $[i+\ell n: k+\ell n]$
with $\ell\in\ZZ$ will be called {\em equivalent}.

\smallskip
An interval $I=[i:k]$ is called {\it short}\ if 
$0\leq k-i< n$. 

Any short interval is equivalent to some interval
$[i:k]$ with $1\leq i\leq n$ and $k\leq 2n-1$.

\smallskip
A short interval $I=[i:k]$ is {\it maximal}\ if 
(using the notation \eqref{MaxFr})
$$
 m_I=M^r\xx(i). 
$$
By Proposition~\ref{prop:len-max-int_le-n} for any $i$
there exists a maximal interval $[i:k]$.

\smallskip
The index $k$ --- the right end of a maximal interval $[i:k]$ --- is called a {\it maximal end}. Being or not a maximal end is a property of the index, not of an interval.

\smallskip
A maximal interval $[i:k]$ is called an {\em $M$-interval}\ (in words, an {\em irreducible maximal interval}) if
for any $k'$ such that $i\leq k'<k$ the interval $[i:k']$
is not maximal. 

By definition, for each $i$, the $M$-interval $[i:k]$ is unique. Put $k=i+\mend(i)$. 
The function $\mend:\;\ZZ \to [0:\,n-1]$ is thus well-defined and $n$-periodic.

\smallskip
Intervals $J=[i:j]$ and $K=[j+1:k]$ are called {\em adjacent}. More precisely, $J$ is left-adjacent to $K$ and $K$ is right-adjacent to $J$. 

\smallskip
A maximal interval of length $n$ is called a {\em full maximal interval}.



\medskip
The average $a_I$ for any interval $I$ of length $n$ equals $\ai{[1:n]}(\xx)$ and will be denoted $\av{\xx}$.

The main result of this section is the following

\begin{prop}
\label{prop:fullmaxint}
There holds the equality
$$
\min_i M^r\xx(i)=\av{\xx}.
$$
In other words, $M^r\xx(i)\geq \av{\xx}$ for $i=1,\dots,n$, and there exists $i_*$ such that $[i_*:\,i_*+n-1]$ is a full maximal interval.
\end{prop}

The inequality $M^r\xx(i)\geq \ai{[i+1:i+n]}(\xx)=\av{\xx}$ is obvious by definition of
the function $M^r\xx(\cdot)$. The nontrivial part is the existence of $i_*$. It is an easy consequence of the
Maximal Ergodic Theorem, see Remark~\ref{rem:MET} below. We give an independent, self-contained proof
with lemmas describing the combinatorics of the set of maximal intervals in full detail.

\smallskip
By a density argument, it suffices to prove 
Proposition~\ref{prop:fullmaxint} under the assumption
$$
\text%
{\bf $x_1,\dots,x_n$ are linearly independent over $\QQ$}.
\eqno(*)
$$

The condition $(*)$ implies that $\av{\xx}$ and the averages $a_I$ for non-equivalent distinct intervals $I$
of lengths $<n$ are all distinct, as well as their integer multiples. It simplifies formulations at some steps of the proof.

\begin{lemma}
\label{lem:adjint}
Suppose $J$ and $K$ are adjacent intervals and $I$ is the disjoint union $I=J\sqcup K$.
If $a_J<a_K$, then $a_J<a_I<a_K$. If $a_J>a_K$, then $a_J>a_I>a_K$. (The same holds true for non-strict inequalities.)
\end{lemma}

\begin{proof}
It follows from the identity
$$
 a_{I}=\frac{|J|}{|J|+|K|} a_J+\frac{|K|}{|J|+|K|} a_K.
\hfill 
\eqno\qedhere
$$
\end{proof}

\begin{lemma}
\label{lem:radjacent-max-int}
Under the assumption $(*)$,
if $J$ is an $M$-interval, $K$ is any interval right-adjacent to it, and $|J|+|K|<2n$, then 
$$
 a_J>a_K.
$$
\end{lemma}

\begin{proof}
The equality $a_J=a_K$ is excluded by condition $(*)$
and at least one of the two lengths $|J|$ and $|K|$ being
$<n$.
The inequality $a_J<a_K$ would contradict the maximality of the interval $J$ by Lemma~\ref{lem:adjint}.
\end{proof}

\begin{lemma}
\label{lem:max-int-nonoverlap}
Under the assumption $(*)$, $M$-intervals  do not overlap. That is, if $I=\intv{i}{k}$ and $I'=\intv{i'}{k'}$ are two $M$-intervals,
then either $I\cap I'=\emptyset$ or one of the intervals contains the other. 
%
\end{lemma}

\begin{proof}
We must exclude the possibility $i<i'\leq k<k'$. 

Since $I'$ is an $M$-interval, we have 
$
 a_{\intv{i'}{k}}<a_{\intv{i'}{k'}}
$
and, by contrapositive to Lemma~\ref{lem:radjacent-max-int},
$
 a_{\intv{i'}{k}}<a_{\intv{k+1}{k'}}$.

Similarly, since $I$ is an $M$-interval,
we get 
$a_{\intv{i}{i'-1}}<a_{\intv{i'}{k}}$.

Hence $a_{\intv{i}{i'-1}}<a_{\intv{i'}{k}}<a_{\intv{k+1}{k'}}$.

Since $\intv{i}{k}=\intv{i}{i'-1}\cup\intv{i'}{k}$, 
we conclude that
$
 a_{I}<a_{\intv{k+1}{k'}}$,
in contradiction with Lemma~\ref{lem:radjacent-max-int}
and $I$ being an $M$-interval.
\end{proof}

\begin{lemma}
\label{lem:max-int-ordering}
Under the assumption $(*)$, if $I'\subsetneq I$ are two $M$-intervals, then $\ai{I'}>\ai{I}$.
\end{lemma}

\begin{proof} The equality $a_I=a_{I'}$ is excluded by the 
assumtions of Lemma. Suppose that 
$\ai{I'}<\ai{I}$. Let $I=[i:k]$, $I'=[i':k']$.

\smallskip
Case $k'<k$. The interval $[i:k']$ is not maximal, so $\ai{[k'+1:\,k]}>\ai{I}$ by contrapositive to Lemma~\ref {lem:radjacent-max-int}. Hence $\ai{[k'+1:\,k]}>\ai{I'}$.
Then by Lemma~\ref {lem:radjacent-max-int} $\ai{[i':k]}>\ai{I'}$, in contradiction to the maximality of $I'$. 

\smallskip
Case $k'=k$. Then necessarily $i<i'$. The interval $\intv{i}{i'-1}$ is not maximal, hence $\ai{\intv{i}{i'-1}}<\ai{I}$
and $\ai{\intv{i'}{k}}>\ai{I}$. So, again, $\ai{\intv{i'}{k}}>\ai{I'}$
with the same contradiction.
\end{proof}

The results of Lemmas~\ref{lem:max-int-nonoverlap}--\ref{lem:max-int-ordering} can be conveniently interpreted in the language of partially ordered sets (posets). We refer to 
\cite[Ch.~3]{Stanley-I} for relevant definitions.

For the fixed $n$-periodic sequence $\xx$ 
the set $\mathcal{I}$ of all $M$-intervals is partially ordered by set-theoretic inclusion:
$$
 I\plt J\;\;\Leftrightarrow \;\; I\subsetneq J.
$$

Under the assumption $(*)$, Lemma~\ref{lem:max-int-nonoverlap} states that the Hasse diagram for the partial order $\plt$
is a forest (disjoint union of trees). 

The assertion of Lemma~\ref{lem:max-int-ordering} means that the function $\ai{\bullet}:\; I\mapsto \ai{I}$
is an order reversing map from $\mathcal{I}$ to $\RR_+$.

Shifts by multiples of $n$ define the equivalence relation on the set $\mathcal I$. The set $\hat{\mathcal I}$ of equivalence classes contains exactly $n$ elements;
representatives of the equivalence classes can be taken in the form $[i:\,i+\kappa(i)]$, $i=1,\dots,n$, where $\kappa(\cdot)$ is the function defined at the beginning of this section. 
The partial order on the set $\hat{\mathcal{I}}$ (which we denote by the same symbol $\plt$) is defined as follows:
two equivalence classes $\hat I$ and $\hat J$ are comparable and $\hat I\plt \hat J$ if and only if there exist their representatives $I$ and $J$ such that $I\plt J$. 

The above interpretation of Lemmas~\ref{lem:max-int-nonoverlap}--\ref{lem:max-int-ordering} applies also to the poset 
$\hat{\mathcal{I}}$.
 

\begin{lemma}
\label{lem:outer-max-int-uniq}
Under the assumption $(*)$,
there exists a unique, modulo shifts by $n$, full maximal interval.  
\end{lemma}

\begin{corollary}
Let $I=\intv{i_*}{i_*+n-1}$ be the full maximal interval.
Then $\av{\xx}=a_I<a_J$ for any $M$-interval $J$ of length $<n$. In the poset interpretation, the Hasse diagram for the poset $(\hat{\mathcal{I}},\plt)$ is a tree with root (the unique maximal element) $\hat I$.  
\end{corollary}

\begin{proof}
Let $\hat I_1, \hat I_2, \dots, \hat I_k$ be the equivalence classes of maximal elements of the set $\hat{\mathcal{I}}$. 
By Lemma~\ref{lem:max-int-nonoverlap} we can take their representatives to be non-overlapping intervals of total length $n$:
$$
I_1=\intv{i_1}{i_2-1},\quad I_2=\intv{i_2}{i_3-1},\;\;\dots,\;\; I_k=\intv{i_k}{i_1+n-1}.
$$
The $M$-interval $I_{k+1}$ beginning at $i_1+n$ is equivalent to $I_1$.

The interval $I_{j+1}$ is right-adjacent to $I_j$
for $j=1,\dots,k$. By Lemma~\ref{lem:radjacent-max-int}, if $k>1$, then 
$$
 a_{I_1}> a_{I_2}>\dots> a_{I_k}> a_{I_{k+1}}=a_{I_1},
$$
a contradiction. Therefore $k=1$ and $I_1=\intv{i_1}{i_1+n-1}$ is a full maximal interval. 

The existence of two non-equivalent full maximal intervals would contradict Lemma~\ref{lem:max-int-nonoverlap}.
\end{proof}

\medskip
Proposition~\ref{prop:fullmaxint} follows from Lemma~\ref{lem:outer-max-int-uniq}. (The function $\xx\mapsto \min_i M^r\xx(i)$ is continuous, so the case of $\xx$ satisfying $(*)$ implies the general case.)  

\medskip
We conclude this section with two remarks connecting the discussed material to mathematical theories of general interest and a numerical example.

\begin{remark}
\label{rem:MET}
Let us explain the connection of the nontrivial part of Proposition~\ref{prop:fullmaxint} with Maximal Ergodic Theorem (MET). 

Consider the set $\intv{1}{n}$ as the probability space with equal probabilities of states, $p(i)=1/n$. The cyclic shift $T:\;i\mapsto i+1\mod n$ is a measure-preserving transformation. MET (see e.g.\ \cite[Theorem~2.4]{Billingsley_1965}) applied to the funciton $i\mapsto x_i$ asserts that, given any $\lambda\in\RR$, if $\mathcal{N}_\lambda=\{i\mid M^r\xx(i)\geq\lambda\}$,
then 
\begin{equation}
\label{MET}
\lambda\frac{|\mathcal{N}_\lambda|}{n}\leq
\frac{1}{n}\sum_{i\in\mathcal{N}_\lambda}x_i.
\end{equation}
Take $\lambda=\min_i M^r\xx(i)$. Then $\mathcal{N}_\lambda=[1:n]$, hence the right-hand side in \eqref{MET} equals $\av{\xx}$ and we obtain the inequality $\min_i M^r\xx(i)\leq\av{\xx}$.
\end{remark}

\begin{remark}
\label{rem:majorization}
Let $I=[i_*:\,i_*+n-1]$ be a full maximal interval.
Put $A=\av{\xx}$. Suppose $\xx$ satisfies the condition $(*)$.
By Lemma~\ref{lem:radjacent-max-int}, 
for $k=1,2,\dots,n-1$ we have
$$
 \sum_{j=0}^{k-1} x_{i_*+j}< Ak,
$$
while for $k=n$ the inequality turns to equality.
It means that we have the majorization relation \cite{MarshallOlkin}
$$
(x_{i_*},x_{i_*+1},\dots,x_{i_*+n-1})\succ (A,\dots,A).
$$
We have shown that among $n$ non-equivalent cyclic shifts of the sequence $\xx$
there is exactly one for which the stated majorization relation is true.
\end{remark}

\subsubsection*{Numerical example}

Consider the following $n$-tuple with $n=10$:
$$
 \xx=(1.2,\;\, 2.3, \;\,3.5,\;\, 1.8,\;\, 1.6,\;\, 2.4,\;\, 3,\;\, 3.2,\;\, 1.1, \;\,2.5).
$$

The definition of $M$-intervals $\intv{i}{i+\kappa(i)}$
does not require the condition $(*)$.
Of course, it is not satisfied here, yet the uniqueness
of maximal intervals $\intv{i}{k}$ is in place
and all relevant terminology is applicable.


We have 
$$
 \av{\xx}=\ai{\intv{1}{10}}=2.26.
$$

The averages $\ai{\intv{i}{i+r-1}}$ are given in the next table. The row with $r=10$ is omitted, as it would contain
the constant $\av{\xx}$ in all cells.
The boldface values correspond to $M$-intervals.

The column with $i=9$ is marked with asterisk.
In it, the omitted value $\ai{\intv{9}{18}}$ is 
maximal and the interval $\intv{9}{18}$ is the
unique full maximal interval.

\bigskip
\noindent\hspace*{-0.3em}
\begin{tabular}{c||c|c|c|c|c|c|c|c|c|c}
$r \backslash i$ & 1 & 2 & 3 & 4 & 5 & 6 & 7 & 8 & $9^*$ & 10
\\
\hline 
1 & 1.2   & 2.3   & {\bf 3.5}   & 1.8   & 1.6   & 2.4   & {\bf 3}     & {\bf 3.2}   & 1.1   & {\bf 2.5}  \\
\hline 
2 & 1.75  & {\bf 2.9}   & 2.65  & 1.7   & 2     & 2.7   & 3.1   & 2.15  & 1.8   & 1.85  \\
\hline 
3 & 2.333 & 2.533 & 2.3   & 1.933 & 2.333 & {\bf 2.867} & 2.433 & 2.267 & 1.6   & 2     \\
\hline 
4 & 2.2   & 2.3   & 2.325 & 2.2   & {\bf 2.55}  & 2.425 & 2.45  & 2     & 1.775 & 2.375 \\
\hline 
5 & 2.08  & 2.32  & 2.46  & {\bf 2.4}   & 2.26  & 2.44  & 2.2   & 2.06  & 2.12  & 2.26  \\
\hline 
6 & 2.133 & 2.433 & 2.583 & 2.183 & 2.3   & 2.233 & 2.217 & 2.3   & 2.067 & 2.15  \\
\hline 
7 & 2.257 & 2.543 & 2.371 & 2.229 & 2.143 & 2.243 & 2.4   & 2.229 & 2     & 2.186 \\
\hline 
8 & {\bf 2.375} & 2.363 & 2.388 & 2.1   & 2.163 & 2.4   & 2.325 & 2.15  & 2.05  & 2.288 \\
\hline 
9 & 2.233 & 2.378 & 2.256 & 2.12  & 2.311 & 2.333 & 2.244 & 2.178 & 2.156 & 2.389 \\
\hline 
\end{tabular}

\bigskip

The poset $\hat{\mathcal{I}}$ as a collection of intervals $\intv{i}{i+\kappa(i)}$, $i=1,\dots,10$, is schematically presented below
(with indices $\!\!\mod 10$).

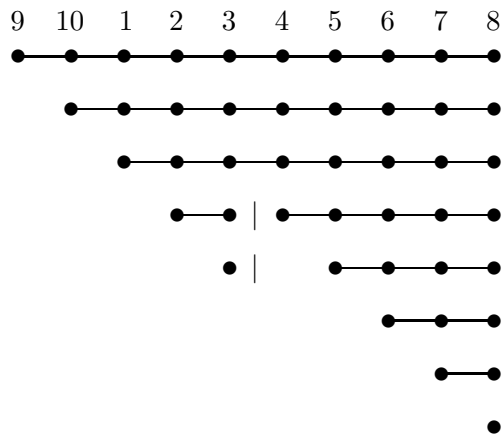
\begin{center}
{\small
\begin{picture}(200,170)
\put(-3,150){$9$}
\put(14,150){$10$}
\put(38,150){$1$}
\put(57,150){$2$}
\put(77,150){$3$}
\put(97,150){$4$}
\put(117,150){$5$}
\put(137,150){$6$}
\put(157,150){$7$}
\put(177,150){$8$}
\multiput(0,140)(20,0){10}{\circle*{5}}
\put(0,140){\line(1,0){180}}
\multiput(20,120)(20,0){9}{\circle*{5}}
\put(20,120){\line(1,0){160}}
\multiput(40,100)(20,0){8}{\circle*{5}}
\put(40,100){\line(1,0){140}}
\multiput(60,80)(20,0){7}{\circle*{5}}
\put(88,77){$|$}
\put(60,80){\line(1,0){20}}
\put(100,80){\line(1,0){80}}
\put(80,60){\circle*{5}}
\multiput(120,60)(20,0){4}{\circle*{5}}
\put(88,57){$|$}
\put(120,60){\line(1,0){60}}
\multiput(140,40)(20,0){3}{\circle*{5}}
\put(140,40){\line(1,0){40}}
\multiput(160,20)(20,0){2}{\circle*{5}}
\put(160,20){\line(1,0){20}}
\put(180,0){\circle*{5}}
\end{picture}
}
\end{center}

The maximal element is the interval $\intv{9}{18}$.

The minimal elements are the intervals $\intv{3}{3}$ and $\intv{8}{8}$
corresponding to two maximal ends, $3$ and $8$.

\section{Uncycling}
\label{sec:uncycling}

\subsection{Auxiliary non-cyclic problem}
\label{ssec:noncyclic}

Let $\xx=(x_i)$, $-\infty<i\leq 0$, be a sequence of nonnegative real numbers with only finitely many nonzero terms.
When $i<0$, we define $\mf{i}(\xx)$ by the formula \eqref{mfavg} with understanding that $r\leq |i|$. To set up the problem which is the subject of this section we do not need to define $\mf{0}(\xx)$.  

Given $p>0$, define
$$
 T(\xx,p)=\sum_{i\leq -1} \frac{x_i}{\mf{i}(\xx)}+\frac{x_0}{p}.
$$
Clearly, this function is homogeneous of order $0$, that is, for any $t>0$
$$
T(t\xx, tp)=T(\xx,p).
$$

For $N=1,2,\dots$, let $\Delta_N$ be the set of nonnegative sequences 
$\xx=(x_i)_{i=-\infty}^0$ such that
$x_i=0$ for $i\leq -N$ and $\sum_{i=1-N}^0 x_i=1$.
We will also treat members of $\Delta_N$ as finite
$N$-tuples $(x_{1-N},\dots,x_0)$. Where it is desireable to emphasize this point of view we, will write $T_N(\xx,p)$ instead of $T(\xx,p)$.

Define
\begin{equation}
\label{infTN}
\minT{N}(p)=\inf_{\xx\in\Delta_N} T(\xx,p).
\end{equation}

Since $\Delta_{N}\subset\Delta_{N+1}$, for every $p$ and $N\geq 1$ we have
$
\minT{N+1}(p)\leq \minT{N}(p)
$.
Therefore there exists a monotone limit
$$
\infT(p)=\downarrow\lim_{N\to\infty} \minT{N}(p).
$$

The relevance of the described problem \eqref{infTN} is explained by the following 

\begin{prop}
\label{prop:noncyclic-to-cyclic}
For any $n\geq 1$ the minimum value \eqref{infS} in the cyclic problem is equal to the corresponding minimum value for the non-cyclic problem \eqref{infTN} with $N=n$ and $p=1/n$
and, moreover, to the limit value $\infT(1/n)$:
\begin{equation}
\label{max-uncyc-cyc}
\infS{n}=\minT{n}(1/n)=\infT(1/n).
\end{equation} 
\end{prop}

We begin with three lemmas which set the direction of the proof. The remaining steps will follow under their thematic headings. 

\begin{lemma}
\label{lem:minexistence}
The lower bound in the right-hand side of \eqref{infTN} is attainable, that is, there exists an $N$-tuple $\xx^{(N)}\in\Delta_N$
such that
$$
 T(\xx^{(N)},p)=\min_{\xx\in\Delta_N} T(\xx,p).
$$
Any such vector $\xx^{(N)}$ will be called a 
$(N,p)$-minimizer.
\end{lemma}

(This lemma parallels Proposition~\ref{prop:len-max-int_le-n}.)

\begin{proof}
The map $\xx\mapsto T_N(\xx,p)$, $\Delta_N\to[0,+\infty]$ is continuous and not identically equal to $\infty$. Since $\Delta_N$ is compact, the finite minimum value is attained. 
\end{proof}

Based on the results of Section~\ref{sec:maxint}, it is easy to establish the analog of the relation \eqref{max-uncyc-cyc} with ``$\geq$'' signs.

\begin{lemma}
\label{lem:uncycling-gt}
There holds the inequality 
$\infS{N}\geq \minT{N}(1/N)$.
\end{lemma}

\begin{proof}
Consider an $N$-periodic nonnegative sequence $\xx$ satisfying the condition
$(*)$ of Section~\ref{sec:maxint}, so as to ensure that appropriate results of Section~\ref{sec:maxint} are applicable. 

Without loss of generality (using a cyclic renumbering if needed) 
we may assume that $\intv{1-N}{0}$ is a full maximal interval.
Suppose also that the normalization condition $\sum_{i=1-N}^0 x_i=1$ holds.
Clearly, $\inf S_N(\xx)=\infS{N}$ where $\xx$ runs over the set of sequences satisfying the stated conditions. 




Now let us treat the $N$-tuple $(x_{1-N},\dots,x_0)$
as an element of $\Delta_N$ in the non-cyclic problem.
The full maximal interval in the cyclic problem being $\intv{1-N}{0}$ implies that for $1-N\leq i\leq-1$ the values $m^+_i(\xx)$ are identical in the cyclic and non-cyclic case. 


The assumption $\sum_{i=1-N}^0 x_i=1$ corresponds to the equality $m^+_{0}(\xx)=\av{\xx}=1/N$ in the cyclic problem. The term $x_0/m^+_{0}(\xx)$ in the definition of $S_N(\xx)$ corresponds to the term $x_0/(1/N)$ in the definition of $T_N(\xx,1/N)$.

We conclude that $S_N(\xx)=T_{N}(\xx,1/N)$.
Hence $S_N(\xx)\geq \minT{N}(1/N)$.
The claimed inequality follows by
taking $\inf_{\xx}$ in the left-hand side.
\end{proof}

The next Lemma shows that a certain structure of a $(N,p)$-minimizer implies the converse inequality and hence  Proposition~\ref{prop:noncyclic-to-cyclic}. 

\begin{lemma}
\label{lem:uncycling-conditional}
Suppose that $\xx\in\Delta_{N}$ is a $(N,1/N)$-minimizer  
such that 

\smallskip
{\rm(i)} for some $k\in\intv{2-N}{0}$
\begin{equation}
\label{monot-minimizer}
  x_{k}\geq x_{1+k}\geq\dots\geq x_{0}\geq \frac{1}{N};
\end{equation}

{\rm(ii)} if $k\geq 3-N$, then 
$$
x_{1-N}=\dots=x_{k-2}=0.
$$
{\em (Note: the conditions (i), (ii) do not impose any restriction on $x_{k-1}$.)}

\smallskip
Then $\infS{N}= \minT{N}(1/N)$. 
\end{lemma}

\begin{proof}
We claim that 
\begin{equation}
\label{ai-conditional}
\ai{[1-N:\,j]}\leq 1/N\quad\text{for $j=1-N,\dots,-1$}.
\end{equation}
For $j\leq k-2$ this is trivial, since the left-hand side
equals $0$. Suppose that $j\geq k-1$ and $\ai{\intv{1-N}{j}}>1/N$.
The assumption \eqref{monot-minimizer} implies
$\ai{\intv{j+1}{0}}\geq 1/N$. Hence we get (cf.~Lemma~\ref{lem:adjint}) $1/N=\ai{\intv{1-N}{0}}>1/N$, a contradiction.


The inequalities \eqref{ai-conditional}
mean that $\intv{1-N}{0}$ is a full maximal interval in the cyclic problem. Hence, as in the proof of Lemma~\ref{lem:uncycling-gt}, we have the equality of
the ``cyclic'' and ``non-cyclic'' sums: $S_N(\xx)=T_{N}(\xx,1/N)$. Given that
$T_N(\xx,1/N)=\minT{N}(1/N)$, we obtain
$\infS{N}\leq \minT{N}(1/N)$. 
%
\end{proof}

In order to prove Proposition~\ref{prop:noncyclic-to-cyclic} we will
study the structure of a $(N,p)$-minimizer and verify the conditions (i) and (ii) of Lemma~\ref{lem:uncycling-conditional}.

\subsection{The structure of a minimizer}

We begin with a simple a technical lemma
concerning the relation of the right maximal averages
\eqref{MaxFr} of two majorization-comparable vectors
(cf.\ Remark~\ref{rem:majorization}).
It will be used in the proof of Lemma~\ref{lem:maxseq-monotone}

\begin{lemma}
\label{lem:merging}
Let $\xi=(\xi_1,\dots,\xi_r)$ and $\eta=(\eta_1,\dots,\eta_r)$ be two nonnegative $r$-tuples.
If $\xi\succeq\eta$ in the sense of majorization, i.e. $\ai{\intv{1}{i}}(\xi)\geq \ai{\intv{1}{i}}(\eta)$ for $i=1,\dots,r-1$, and $\ai{\intv{1}{r}}(\xi)= \ai{\intv{1}{r}}(\eta)$, then $M^r\xi(1)\geq M^r\eta(1)$.
%
\end{lemma}

\begin{proof}
Pointwise comparison 
of the definitions
$M^r\xx(1)=\max_i \ai{\intv{1}{i}}(\xx)$ for $\xx=\xi$
and $\xx=\eta$ makes the claim obvious.
%
\end{proof}

\begin{lemma}
\label{lem:maxseq-monotone}
Let $N\geq 1$ and $\xx\in\Delta_N$ be a $(N,p)$-minimizer.
Suppose that $x_{k-1}\neq 0$ for some $k\in\intv{2-N}{0}$.
Then 
$$
 x_{k}\geq x_{1+k}\geq\dots\geq x_{0}\geq p.
$$
(In other words, the minimizing sequence on its support, augmented by $p$ on the right, is monotone except possibly at the leftmost term.)
\end{lemma}

\begin{proof}
Suppose, by way of contradiction, that at least one of the claimed inequalities is violated. Let $x_j$ be the rightmost term for which $x_{j}<x_{j+1}$; if $x_{0}<p$
we put $j=0$.
%
Define $\yy\in\Delta_{N-1}$ as follows:
$$
   y_i=
\begin{cases}
x_i\quad\mbox{\rm for $j+1\leq i\leq 0$},\\
   x_{j-1}+x_{j} \quad\mbox{\rm for $i=j$},\\
   x_{i-1}\quad\mbox{\rm for $2-N\leq i\leq j-1$}
\end{cases}
.
$$
(The cases $i\geq j+1$ or $i\leq j-1$ may happen to be void.)


We will show that $T_{N-1}(\yy,p)<T_N(\xx,p)$.
The definition of $\yy$ implies that
$T_N(\xx,p)-T_{N-1}(\yy,p)=R_1+R_2$, where
$$
R_1=\sum_{i=1-N}^{j-2} x_{i}\left(\frac{1}{m^+_{i}(\xx)}
-\frac{1}{m^+_{i+1}(\yy)}\right)
$$
and
$$
 R_2=\begin{cases}\dst
   x_{j-1}\left(\frac{1}{m^+_{-1}(\xx)}
-\frac{1}{p}\right)
&\quad\mbox{\rm if $j=0$},\\[2ex]
\dst
   x_{j-1}\left(\frac{1}{m^+_{j-1}(\xx)}
-\frac{1}{m^+_{j}(\yy)}\right)
&\quad\mbox{\rm if $j<0$}.
\end{cases}
$$

Let us show first that $R_1\geq 0$ and $R_2\geq 0$.

For $R_1$ it follows by Lemma~\ref{lem:merging}
applied successively, for every $i\in\{2-N,\dots,j-1\}$,
to $\eta=(x_i,\dots,x_0)$ and $\xi=(y_{i+1},\dots,y_0,0)$.

For $R_2$ we have: if $j=0$, then $p>x_0=m^+_{-1}(\xx)$
by assumption. If $j<0$, then $m^+_{j}(\yy)=m^+_{j}(\xx)=x_{j+1}$
(since $x_{j+1}\geq\dots\geq x_0$). Now, by assumption,
$x_{j}<x_{j+1}$ and it follows that $m^+_{j-1}(\xx)<x_{j+1}$.

In the case $x_{j-1}>0$ we have the strict inequality $R_2>0$, which immediately yields $T_{N-1}(\yy,p)<T_N(\xx,p)$.

Suppose that $x_{i}=0$ for $\ell+1\leq i\leq j-1$,
while $x_\ell>0$. Such an $\ell$ exists because we are given that $x_{k-1}>0$. Consider the difference 
$$
\Delta=\frac{1}{m^+_{\ell}(\xx)}
-\frac{1}{m^+_{\ell+1}(\yy)}.
$$
Once we prove that $\Delta>0$, the strict inequality
$R_1>0$ will follow and we obtain $T_{N-1}(\yy,p)<T_N(\xx,p)$ in this case, too.

Let $\intv{\ell+1}{k}$ be the $M$-interval at $\ell+1$ for $\xx$, so
$m^+_\ell(\xx)=\ai{\intv{\ell+1}{k}}(\xx)$.
Then $k\geq j+1$, since $x_{\ell+1}=\dots=x_{j-1}\leq x_j<x_{j+1}$. Therefore
$$
m^+_{\ell+1}(\yy)\geq 
\ai{\intv{\ell+2}{k}}(\yy)=\frac{r}{r-1}\ai{\intv{\ell+1}{k}}(\xx)>m^+_\ell(\xx),
$$
where $r=k-\ell$.

Thus, in all cases $T_{N-1}(\yy,p)<T_N(\xx,p)$,
hence $\xx$ cannot be a $(N,p)$-minimizer.
The proof is complete.
\end{proof}

\begin{remark} 
\label{rem:p1}
For $p\geq 1$ the $(N,p)$ minimizer is the sequence of the form $[0,\dots,0,1]$. Indeed, if $x_{-k}\neq 0$ for some $k>0$, then by the above lemma $x_0\geq p\geq 1$. On the other hand, 
$x_0\leq 1-x_{-k}<1$, a contradiction.
\end{remark}

\subsection{End of cyclic-to-noncyclic reduction}

\begin{lemma}  
\label{lem:maxseq-stabilization}
The nonincreasing sequence $(\minT{N}(p))$ stabilizes no later than at $N=\lceil 1/p\rceil$, that is,
$\infT(p)=\minT{N}(p)$ for any integer $N\geq \lceil 1/p\rceil$.
\end{lemma}

\begin{proof}
If $\xx^*$ is a $(N,p)$-maximizer and $x^*_{-k}\neq 0$, then by Lemma~\ref{lem:maxseq-monotone} we have
 $1\geq \sum_{i=1-k}^0 x_i\geq kp$. 
Hence $k\leq 1/p$, so $x^*_i=0$ for $i< -1/p$. It follows that the solution of the extremal problem 
\eqref{infTN} is independent of $N$ for 
$N\geq \lceil 1/p\rceil $.
\end{proof}

\begin{remark}
For $p\geq 1$ the result is sharp: the stabilization occurs at $N=1$ (i.e.\ immediately), in agreement with Remark~\ref{rem:p1}.

If $p$ is small, then in fact the stabilization occurs much earlier than what Lemma~\ref{lem:maxseq-stabilization} promices, namely at $N=O(|\log p|)$. It follows from the analysis of recurrence relations associated with reduced optimization problem (see next section) given in 
\cite{Sadov_2022logper}.
\end{remark}
 
\begin{proof}[Proof of Proposition~\ref{prop:noncyclic-to-cyclic}]
The condition \eqref{monot-minimizer} of Lemma~\ref{lem:uncycling-conditional} is fulfilled by Lemmas~\ref{lem:maxseq-stabilization} and~\ref{lem:maxseq-monotone}.
\end{proof}

\section{The reduced optimization problem}
\label{sec:reduction}

Recall the functions already mentioned in Section~\ref{sec:summfa}
(in comparison with Eqs.~\eqref{tTN} and \eqref{mintTN}
here we shift the indices by $-N$)
$$
 \tT_N(\xx,p)=\sum_{i=1-N}^{-1} \frac{x_i}{x_{i+1}}+\frac{x_0}{p}, \qquad \xx\in\Delta_N,
$$
and
$$
\mintT{N}(p)=\inf_{\xx\in\Delta_N} \tT_N(\xx,p).
$$

We consider the minimization problem 
\eqref{infTN} in the non-cyclic case with an arbitrary $p<1$. (The case $p\geq 1$, although not relevant to the proof of Proposition~\ref{prop:reduction}, is covered by Remark~\ref{rem:p1}; it is shown to be trivial.)

\begin{lemma} 
\label{lem:noncyc-simplification}
For any $N\geq 1$ and any $p>0$ there holds the identity
$$
 \mintT{N}(p)=\minT{N}(p).
$$
\end{lemma}

\begin{proof} 
Since $m^+_i(\xx)\geq x_{i+1}$, we have $T_N(\xx,p)\leq \tT_N(\xx,p)$ for any $\xx\in\Delta_N$.

On the other hand, let $\xx^*$ be a $(N,p)$-minimizer. Then by Lemma~\ref{lem:maxseq-monotone} $m^+_i(\xx^*)= x^*_{i+1}$, hence 
$$
\minT{N}(p)= T_N(\xx^*,p)= \tT_N(\xx^*,p)\geq\mintT{N}(p).
\hfill\eqno{\qedhere}
$$
\end{proof}

In view of Proposition~\ref{prop:noncyclic-to-cyclic}
and Lemma~\ref{lem:noncyc-simplification}, the proof of Proposition~\ref{prop:reduction} is complete.

\end{document}